\documentclass[10pt,a4paper,oneside]{article}
\usepackage[a4paper,left=1.3in,right=1in,top=1.3in,bottom=1in]{geometry}
\usepackage{amsmath,amsthm,pstricks-add,enumerate}
\usepackage[dvips]{graphicx}
\usepackage{color}
\usepackage{mathtools}
\usepackage{hyperref}
\usepackage{amssymb,color}
\usepackage{cite}
\usepackage{pdfsync}
\usepackage{orcidlink}
\usepackage{hyperref}
\usepackage{amsmath,amsthm,amssymb,enumerate}

\allowdisplaybreaks
\normalsize
\newtheoremstyle{thm}{\baselineskip}{\baselineskip}{\slshape}{}{\bfseries}{.}{ }{ }

\theoremstyle{thm}
\newtheorem{thm}{Theorem}[section]
\newtheorem{cor}[thm]{Corollary}

\theoremstyle{remark}

\begin{document}
	\title{Degree–Adjacency–Diametrical Matrix of Graphs: Spectral Properties and Energy Bounds}
\author{
	S. Celine Prabha$^{1}$, Dhivviyanandam I $^{2*}$\\
	$^{1}$Department of Mathematics\\
	SRM Institute of Science and Technology\\
	Kattankulathur, Chengalpattu--603203, Tamil Nadu, India\\
	\texttt{celineprabha.s@gmail.com}\\
    orcid id: 0000-0002-4621-2450\\
	$^{2*}$Vatican Observatory, Specola Vaticana\\
	V-00120, Vatican City State\\
	\texttt{dhivviyanandam@specola.va}\\
    orcid Id: 0000-0002-3805-6638\\
	}

\date{}
\maketitle

\begin{abstract}
	Graph matrices that combine structural and metric information have attracted considerable attention in spectral graph theory. In this paper, we introduce the \emph{Degree--Adjacency--Diametrical (DAD) matrix} of a connected graph $G$. The DAD matrix is defined by
	\[
	(DAD(G))_{ij}=
	\begin{cases}
		d(v_i), & \text{if } i=j,\\
		1, & \text{if } v_i\sim v_j,\\
		\operatorname{diam}(G), & \text{if } d(v_i,v_j)=\operatorname{diam}(G),\\
		0, & \text{otherwise},
	\end{cases}
	\]
	where $d(v_i)$ denotes the degree of the vertex $v_i$. We establish fundamental spectral properties of the DAD matrix by deriving trace identities, characteristic polynomials, and spectra for several classes of graphs. Furthermore, we obtain a spectral expression for the Wiener index of graphs with diameter two, derive upper bounds for the largest DAD eigenvalue, and establish lower bounds for the DAD energy. Explicit spectral characterizations are also obtained for complete multipartite graphs. The proposed matrix provides a unified spectral framework that simultaneously incorporates degree, adjacency, and diametrical information, thereby enriching the study of graph spectra and introducing a new graph matrix for spectral graph theory.
\end{abstract}



\section{Introduction}

Spectral graph theory is an important branch of graph theory that studies the
relationship between the structural properties of graphs and the eigenvalues of
matrices associated with graphs. Various graph matrices, such as adjacency
matrices, Laplacian matrices, signless Laplacian matrices, and distance matrices,
have been extensively investigated due to their applications in chemistry,
network analysis, communication systems, and combinatorial optimization. The
eigenvalues of these matrices, known as graph spectra, provide significant
information about the structural characteristics of graphs. A comprehensive
introduction to the theory of graph spectra and its applications can be found in
the works of Cvetkovi\'c, Rowlinson, and Simi\'c~\cite{cvetkovic2010} and
Bapat~\cite{bapat2010}.

Among the different graph matrices, the distance matrix has received
considerable attention because it incorporates global structural information of
a graph. For a connected graph $G$ with vertex set
$V(G)=\{v_1,v_2,\ldots,v_n\}$, the distance matrix $D(G)$ is defined as the
matrix whose $(i,j)$-entry represents the distance between the vertices
$v_i$ and $v_j$. Unlike adjacency matrices, which describe local connections,
distance matrices capture the overall metric structure of graphs. The spectral
properties of distance matrices have been extensively studied, and several
fundamental results on distance spectra are summarized in the survey by
Aouchiche and Hansen~\cite{aouchiche2014}. The development of Laplacian-type
matrices associated with distance matrices has further expanded the study of
distance-based spectral parameters~\cite{aouchiche2013}.

The investigation of distance spectra has led to numerous results concerning
distance eigenvalues, distance energy, cospectral graphs, and relationships
between structural parameters and eigenvalues. Indulal studied the distance
spectra of graph compositions and obtained spectral characterizations for
several graph operations~\cite{indulal2009}. Stevanovi\'c and Indulal
investigated the distance spectrum and energy of compositions of regular graphs,
establishing important connections between graph structure and distance
eigenvalues~\cite{stevanovic2009}. Further developments in distance spectral
theory, including generalized distance matrices and their applications, have
been presented by DeVille~\cite{deville2022}. Recent surveys on variants of
distance matrices and their spectral properties demonstrate the continuing
importance of distance-based graph matrices in modern spectral graph
theory~\cite{hogben2022}.

Several researchers have also focused on spectral properties of special graph
classes. Stevanovi\'c and Ili\'c provided a detailed study of the spectral
properties of distance matrices and their applications in molecular graph
theory~\cite{stevanovic2012}. The relationship between distance spectra,
transmission properties, and distance cospectrality has been explored in recent
studies, including the works of Aalipour \textit{et al.}~\cite{aalipour2016}
and Abiad \textit{et al.}~\cite{abiad2017}. In addition, distance spectra of
distance-regular graphs have been investigated by Atik and Panigrahi~\cite{atik2015},
providing further insight into the influence of graph structure on distance
eigenvalues. The classical work of Graham and Pollak~\cite{graham1971} also
established fundamental results related to distance matrices, motivating many
subsequent investigations in distance-based spectral theory.

Although distance matrices and their variants successfully capture global
structural information, they do not simultaneously incorporate local vertex
information such as vertex degrees and adjacency relationships. On the other
hand, adjacency and degree-based matrices mainly describe local structural
properties and may not fully represent the global metric characteristics of
graphs. Therefore, combining local and global structural information in a
single matrix provides a promising direction for developing new graph
invariants with richer spectral properties.

The integration of local and global structural information in a single graph matrix has the potential to provide more informative spectral descriptors for complex networks. Such hybrid matrices are useful in areas such as mathematical chemistry, network analysis, communication systems, and graph machine learning, where both connectivity and distance characteristics influence the underlying graph structure. Motivated by these applications, the development of new graph matrices that unify multiple structural features has become an important direction in spectral graph theory.

Recently, Leka Amruthavarshini and Rajkumar introduced the
adjacency--diametrical matrix (AD matrix), which combines adjacency
information and diametrical relationships of a graph~\cite{leka2026}. The AD
matrix provides a framework for studying the spectral properties of a graph by
incorporating both local adjacency and maximum distance information. Inspired
by this idea, we further explore a graph matrix that incorporates vertex degree
information along with adjacency and diametrical relationships, leading to the
introduction of the Degree--Adjacency--Diametrical (DAD) matrix.

We now introduce a new graph matrix that combines degree, adjacency, and
diametrical information. Let $G$ be a connected graph with vertex set
$V(G)=\{v_1,v_2,\ldots,v_n\}$ and diameter $\operatorname{diam}(G)$. The
\textit{Degree--Adjacency--Diametrical (DAD) matrix} of $G$, denoted by
$DAD(G)=[d_{ij}]_{n\times n}$, is defined as

\[
d_{ij}=
\begin{cases}
	\deg(v_i), & \text{if } i=j,\\[4pt]
	1, & \text{if } v_i \text{ and } v_j \text{ are adjacent},\\[4pt]
	\operatorname{diam}(G), & \text{if } d(v_i,v_j)=\operatorname{diam}(G),\\[4pt]
	0, & \text{otherwise}.
\end{cases}
\]

The proposed matrix integrates three fundamental structural features of a
graph: vertex degree, local adjacency, and maximum distance relationships.
Hence, the DAD matrix provides a unified framework that combines local
connectivity and global extremal distance information.

Throughout this paper, all graphs considered are finite, simple, and connected.
For a graph $G=(V,E)$, the degree of a vertex $v\in V(G)$ is denoted by
$d(v)$. The distance between two vertices $u$ and $v$, denoted by
$d(u,v)$, is the length of a shortest path connecting them. The diameter of
$G$, denoted by $\operatorname{diam}(G)$, is defined as

\[
\operatorname{diam}(G)=
\max\{d(u,v):u,v\in V(G)\}.
\]

The characteristic polynomial of a square matrix $M$ of order $n$ is defined by
\[
\phi(M,\lambda)=\det(\lambda I_n-M),
\]
where $I_n$ denotes the identity matrix of order $n$. The roots of
$\phi(M,\lambda)$ are called the \emph{eigenvalues} of $M$. The
\emph{spectrum} of $M$, denoted by $\operatorname{Spec}(M)$, is the multiset of
all eigenvalues of $M$.

If $\lambda_1,\lambda_2,\ldots,\lambda_n$ are the eigenvalues of $M$, then the
\emph{energy} of $M$ is defined as
\[
E(M)=\sum_{i=1}^{n}|\lambda_i|.
\]

Another important distance-based graph invariant is the
\emph{Wiener index} of a connected graph $G$, introduced by Graham and
Pollak~\cite{graham1971}, which is defined by
\[
W(G)=\sum_{\{u,v\}\subseteq V(G)} d(u,v),
\]
where the summation is taken over all unordered pairs of distinct vertices
$u,v\in V(G)$.

	\section{Main Results}
	\begin{thm}
		Let $P_n$ be the path graph on $n \ge 2$ vertices, and let $DAD(P_n)$ denote its Degree–Adjacency–Diametrical matrix. Then
		\[
		\Phi_{DAD(P_n)}(x) = 
		\begin{cases} 
			x(x-2), & n=2, \\[6pt] 
			x\Phi_{n-1}(x-2) - (n-1)^2\Phi_{n-2}(x-2) + 2(1-n), & n \ge 3, 
		\end{cases}
		\]
		where
		\[
		\Phi_n(x) = \phi_{A(P_n)}(x) = \prod_{k=1}^{n}\left(x-2\cos\frac{\pi k}{n+1}\right),
		\]
		and $\Phi_0(x)=1$.
	\end{thm}
	
	\begin{proof}
		We consider two cases.
		
		\textbf{Case 1: $n = 2$} \\
		The graph $P_2$ consists of two adjacent vertices. Each vertex has degree $1$, and the diameter of the graph is also $1$. Hence
		\[
		DAD(P_2) = \begin{pmatrix} 1 & 1 \\ 1 & 1 \end{pmatrix}.
		\]
		Therefore
		\[
		\Phi_{DAD(P_2)}(x) = \det\begin{pmatrix} x-1 & -1 \\ -1 & x-1 \end{pmatrix} = (x-1)^2 - 1 = x^2 - 2x = x(x-2).
		\]
		Thus the stated formula holds for $n=2$.
		
		\vspace{1em}
		\textbf{Case 2: $n \ge 3$} \\
		For $n \ge 3$, the diameter of $P_n$ is $n-1$. Relabel the vertices as
		\[
		v_1, v_n, v_2, \ldots, v_{n-1}.
		\]
		Since the end vertices have degree $1$ and every internal vertex has degree $2$,
		\[
		DAD(P_n) = \begin{pmatrix} 
			1 & n-1 & 1 & \cdots & 0 & 0 \\ 
			n-1 & 1 & 0 & \cdots & 0 & 1 \\ 
			1 & 0 & 2 & \cdots & 0 & 0 \\ 
			\vdots & \vdots & \vdots & \ddots & \vdots & \vdots \\ 
			0 & 0 & 0 & \cdots & 2 & 1 \\ 
			0 & 1 & 0 & \cdots & 1 & 2 
		\end{pmatrix}.
		\]
		Hence
		\[
		\Phi_{DAD(P_n)}(x) = \det(xI_n - DAD(P_n)).
		\]
		Write
		\[
		xI_n - DAD(P_n) = \begin{pmatrix} B & C \\ C^T & T \end{pmatrix},
		\]
		where
		\[
		B = \begin{pmatrix} x-1 & 1-n \\ 1-n & x-1 \end{pmatrix},
		\]
		$C$ is the $2 \times (n-2)$ matrix having $-1$ in positions $(1,1)$ and $(2,n-2)$ and zeros elsewhere, and
		\[
		T = (x-2)I_{n-2} - A(P_{n-2}).
		\]
	By the Schur complement formula,
	\[
	\Phi_{DAD(P_n)}(x) = \det(T) \det\left(B - C T^{-1} C^T\right).
	\]
	Since $\det(T) = \Phi_{n-2}(x-2)$, we have:
	\[
	T^{-1} = \frac{\operatorname{adj}(T)}{\Phi_{n-2}(x-2)}.
	\]
	For the path matrix, the corner and boundary values of the adjugate matrix yield:
	\[
	(\operatorname{adj}(T))_{1,1} = (\operatorname{adj}(T))_{n-2, n-2} = \Phi_{n-3}(x-2),
	\]
	and
	\[
	(\operatorname{adj}(T))_{1, n-2} = (\operatorname{adj}(T))_{n-2, 1} = 1.
	\]
	Consequently,
	\[
	C T^{-1} C^T = \frac{1}{\Phi_{n-2}(x-2)} \begin{pmatrix} \Phi_{n-3}(x-2) & 1 \\ 1 & \Phi_{n-3}(x-2) \end{pmatrix}.
	\]
		Substituting (2.3) into (2.1) and using $\det\begin{pmatrix} a & b \\ b & a \end{pmatrix} = a^2 - b^2$, we obtain
		\begin{equation} \tag{2.4}
			\Phi_{DAD(P_n)}(x) = ((x-1)^2 - (n-1)^2)\Phi_{n-2}(x-2) - 2(x-1)\Phi_{n-3}(x-2) + \frac{\Phi_{n-3}^2(x-2) - 1}{\Phi_{n-2}(x-2)} + 2(1-n).
		\end{equation}
		The characteristic polynomial of a path satisfies
		\[
		\Phi_k(y) = U_k\!\left(\frac{y}{2}\right),
		\]
		where $U_k$ is the Chebyshev polynomial of the second kind. The identity
		\[
		U_k^2(t) - U_{k-1}(t)U_{k+1}(t) = 1
		\]
		implies
		\[
		\Phi_{n-3}^2(y) - \Phi_{n-4}(y)\Phi_{n-2}(y) = 1.
		\]
		Hence
		\[
		\frac{\Phi_{n-3}^2(y) - 1}{\Phi_{n-2}(y)} = \Phi_{n-4}(y).
		\]
		Setting $y = x-2$ in (2.4), we obtain
		\begin{equation} \tag{2.5}
			\Phi_{DAD(P_n)}(x) = ((x-1)^2 - (n-1)^2)\Phi_{n-2}(x-2) - 2(x-1)\Phi_{n-3}(x-2) + \Phi_{n-4}(x-2) + 2(1-n). 
		\end{equation}
		The polynomials $\Phi_k$ satisfy the recurrence relation
		\begin{equation} \tag{2.6}
			\Phi_k(y) = y\Phi_{k-1}(y) - \Phi_{k-2}(y). 
		\end{equation}
		Applying (2.6) with $k = n-2$ and $y = x-2$,
		\[
		\Phi_{n-4}(x-2) = (x-2)\Phi_{n-3}(x-2) - \Phi_{n-2}(x-2).
		\]
		Substituting this into (2.5) and collecting like terms yields
		\[
		\Phi_{DAD(P_n)}(x) = (x^2 - 2x - (n-1)^2)\Phi_{n-2}(x-2) - x\Phi_{n-3}(x-2) + 2(1-n).
		\]
		Since $x^2 - 2x = x(x-2)$,
		\[
		\Phi_{DAD(P_n)}(x) = x\Big((x-2)\Phi_{n-2}(x-2) - \Phi_{n-3}(x-2)\Big) - (n-1)^2\Phi_{n-2}(x-2) + 2(1-n).
		\]
		Using the recurrence relation once more,
		\[
		(x-2)\Phi_{n-2}(x-2) - \Phi_{n-3}(x-2) = \Phi_{n-1}(x-2).
		\]
		Therefore
		\[
		\Phi_{DAD(P_n)}(x) = x\Phi_{n-1}(x-2) - (n-1)^2\Phi_{n-2}(x-2) + 2(1-n)
		\]
		for all $n \ge 3$. Combining both cases completes the proof.
	\end{proof}

\begin{thm}
	For $n \geq 3$, let $\text{DAD}(C_n)$ be the Degree--Adjacency--Diametrical matrix of the cycle graph $C_n$. Then:
	
	\textbf{If $n$ is even:}
	\[ \text{Spec}(\text{DAD}(C_n)) = \left\{ 2 + 2\cos\frac{2\pi k}{n} + \frac{n}{2}(-1)^k :\; k=0,1,\ldots,n-1 \right\} \]
	and
	\[ \phi_{\text{DAD}(C_n)}(x) = \prod_{k=0}^{n-1} \left( x - 2 - 2\cos\frac{2\pi k}{n} - \frac{n}{2}(-1)^k \right) \]
	
	\textbf{If $n$ is odd:}
	\[ \text{Spec}(\text{DAD}(C_n)) = \left\{ 2 + 2\cos\frac{2\pi k}{n} + (n-1)\cos\frac{(n-1)\pi k}{n} :\; k=0,1,\ldots,n-1 \right\} \]
	and
	\[ \phi_{\text{DAD}(C_n)}(x) = \prod_{k=0}^{n-1} \left( x - 2 - 2\cos\frac{2\pi k}{n} - (n-1)\cos\frac{(n-1)\pi k}{n} \right) \]
\end{thm}
\begin{proof}
	
	Every vertex of $C_n$ has degree 2. Moreover,
	\[ \text{diam}(C_n) = \begin{cases} \dfrac{n}{2}, & \text{if } n \text{ is even}, \\[2mm] \dfrac{n-1}{2}, & \text{if } n \text{ is odd}. \end{cases} \]
	Hence, $\text{DAD}(C_n)$ is a circulant matrix. Let $\omega_n = e^{2\pi i/n}$ be a primitive $n$-th root of unity. It is well known that if a circulant matrix has the first row $(c_0, c_1, \ldots, c_{n-1})$, then its eigenvalues are:
	\[ \lambda_k = \sum_{j=0}^{n-1} c_j \omega_n^{jk}, \quad k=0, 1, \ldots, n-1. \tag{2.7} \]
	
	\textbf{Case 1: $n = 2m$ (even)} \\
	The first row of $\text{DAD}(C_n)$ is $(2, 1, 0, \ldots, 0, m, 0, \ldots, 0, 1)$. Using (2.1):
	\[ \lambda_k = 2 + \omega_n^k + \omega_n^{-k} + m\omega_n^{mk} \]
	Since $\omega_n^k + \omega_n^{-k} = 2\cos\frac{2\pi k}{n}$ and $\omega_n^{mk} = e^{2\pi i mk/n} = e^{\pi i k} = (-1)^k$, we obtain:
	\[ \lambda_k = 2 + 2\cos\frac{2\pi k}{n} + \frac{n}{2}(-1)^k \]
	Therefore,
	\[ \phi_{\text{DAD}(C_n)}(x) = \prod_{k=0}^{n-1} \left( x - 2 - 2\cos\frac{2\pi k}{n} - \frac{n}{2}(-1)^k \right) \]
	
	\textbf{Case 2: $n = 2m+1$ (odd)} \\
	The first row of $\text{DAD}(C_n)$ is $(2, 1, 0, \ldots, 0, m, m, 0, \ldots, 0, 1)$. Again using (2.1):
	\[ \lambda_k = 2 + \omega_n^k + \omega_n^{-k} + m\omega_n^{mk} + m\omega_n^{(m+1)k} \]
	Since $\omega_n^k + \omega_n^{-k} = 2\cos\frac{2\pi k}{n}$ and $\omega_n^{mk} + \omega_n^{(m+1)k} = 2(-1)^k\cos\frac{\pi k}{n}$, we obtain:
	\[ \lambda_k = 2 + 2\cos\frac{2\pi k}{n} + 2m(-1)^k\cos\frac{\pi k}{n} \]
	Because $n = 2m+1$, we have $2m = n-1$. Using $(-1)^k\cos\frac{\pi k}{n} = \cos\frac{(n-1)\pi k}{n}$, it follows that:
	\[ \lambda_k = 2 + 2\cos\frac{2\pi k}{n} + (n-1)\cos\frac{(n-1)\pi k}{n} \]
	Hence,
	\[ \phi_{\text{DAD}(C_n)}(x) = \prod_{k=0}^{n-1} \left( x - 2 - 2\cos\frac{2\pi k}{n} - (n-1)\cos\frac{(n-1)\pi k}{n} \right) \]
	Combining the two cases completes the proof. $\blacksquare$
\end{proof}

	\begin{thm}
		Let $K_n$ ($n \ge 2$) be the complete graph on $n$ vertices, and let $DAD(K_n)$ denote its Degree–Adjacency–Diametrical matrix. Then the characteristic polynomial of $DAD(K_n)$ is
		\[
		\Phi_{DAD(K_n)}(x) = (x-2n+2)(x-n+2)^{n-1}.
		\]
		Consequently, the spectrum of $DAD(K_n)$ is
		\[
		\operatorname{Spec}(DAD(K_n)) = \{\,2n-2,\, (n-2)^{(n-1)}\,\},
		\]
		where $(n-2)^{(n-1)}$ denotes the eigenvalue $n-2$ with multiplicity $n-1$.
	\end{thm}
	
	\begin{proof}
		Since $K_n$ is a complete graph, every vertex has degree $n-1$, and the diameter of $K_n$ is $1$. By the definition of the Degree–Adjacency–Diametrical matrix, the diagonal entries are the vertex degrees, while every off-diagonal entry is $1$ because every pair of distinct vertices is adjacent (and also at distance equal to the diameter).
		
		Hence
		\[
		DAD(K_n) = \begin{pmatrix}
			n-1 & 1 & 1 & \cdots & 1 \\
			1 & n-1 & 1 & \cdots & 1 \\
			1 & 1 & n-1 & \cdots & 1 \\
			\vdots & \vdots & \vdots & \ddots & \vdots \\
			1 & 1 & 1 & \cdots & n-1
		\end{pmatrix}.
		\]
		
		Let $J_n$ denote the $n \times n$ all-ones matrix. Then
		\[
		DAD(K_n) = (n-2)I_n + J_n.
		\]
		
		Therefore,
		\begin{equation} \tag{2.8}
			xI_n - DAD(K_n) = (x-n+2)I_n - J_n.
		\end{equation}
		
		It is well known that the matrix $J_n$ has eigenvalue $n$ corresponding to the vector
		\[
		\mathbf{1} = (1,1,\ldots,1)^T,
		\]
		and eigenvalue $0$ with multiplicity $n-1$ on the subspace
		\[
		\{\mathbf{x} \in \mathbb{R}^n : \mathbf{1}^T\mathbf{x} = 0\}.
		\]
		
		Consequently, the matrix $DAD(K_n) = (n-2)I_n + J_n$ has eigenvalues
		\[
		(n-2) + n = 2n-2
		\]
		with multiplicity one, and
		\[
		(n-2) + 0 = n-2
		\]
		with multiplicity $n-1$.
		
		Since the characteristic polynomial is the product of $x-\lambda$ over all eigenvalues $\lambda$, we obtain
		\[
		\Phi_{DAD(K_n)}(x) = \det(xI_n - DAD(K_n)) = (x-(2n-2))(x-(n-2))^{n-1}.
		\]
		
		Equivalently,
		\[
		\Phi_{DAD(K_n)}(x) = (x-2n+2)(x-n+2)^{n-1}.
		\]
		
		Hence
		\[
		\operatorname{Spec}(DAD(K_n)) = \{\,2n-2,\, (n-2)^{(n-1)}\,\}.
		\]
		
		This completes the proof.
	\end{proof}

	\begin{thm}[Characteristic Polynomial and Spectrum of $DAD(K_{n_1,\ldots,n_m})$]
		Let $G = K_{n_1,n_2,\ldots,n_m}$ be the complete $m$-partite graph, where $m \ge 2$, $n_i \ge 1$ for all $i \in \{1, \ldots, m\}$, and $n = n_1 + n_2 + \cdots + n_m$. Then the characteristic polynomial of the Degree–Adjacency–Diametrical matrix $DAD(G)$ is
		\[
		\Phi_{DAD(G)}(x) = (x - 2n + 2)(x - n + 2)^{m-1} \prod_{i=1}^m (x - n + n_i + 2)^{n_i-1}.
		\]
		Consequently, the spectrum of $DAD(G)$ consists of the following eigenvalues:
		\begin{itemize}
			\item $n - n_i - 2$ with multiplicity $n_i - 1$ for each $i \in \{1, \ldots, m\}$,
			\item $n - 2$ with multiplicity $m - 1$,
			\item $2n - 2$ with multiplicity $1$.
		\end{itemize}
	\end{thm}
	
	\begin{proof}
	Step 1:
		Let $V_1, V_2, \ldots, V_m$ be the partite sets of $G$, where $|V_i| = n_i$. By the definition of a complete multipartite graph:
		\begin{itemize}
			\item Each vertex in $V_i$ is adjacent to all vertices belonging to $V_j$ for all $j \ne i$.
			\item No two distinct vertices within the same partite set $V_i$ are adjacent.
		\end{itemize}
		Hence, every vertex in $V_i$ has degree $d(v) = n - n_i$. Furthermore, since $m \ge 2$, any two distinct vertices in the same part share common neighbors in the other parts, meaning their distance is exactly $2$. Thus, $\operatorname{diam}(G) = 2$.
		
		The matrix $DAD(G)$ can be partitioned into blocks based on the vertex sets $V_1, \ldots, V_m$. The diagonal block corresponding to $V_i$ is given by
		\[
		(n - n_i)I_{n_i} + 2(J_{n_i} - I_{n_i}) = (n - n_i - 2)I_{n_i} + 2J_{n_i},
		\]
		where $I_{n_i}$ is the identity matrix and $J_{n_i}$ is the all-ones matrix of order $n_i$. Every off-diagonal block representing the adjacencies between $V_i$ and $V_j$ ($i \ne j$) is simply the all-ones matrix $J_{n_i \times n_j}$.

	Step 2: For each $i \in \{1, \dots, m\}$, consider the subspace
	\[
	W_i = \left\{ x \in \mathbb{R}^n : \operatorname{supp}(x) \subseteq V_i \text{ and } \sum_{v \in V_i} x_v = 0 \right\}.
	\]
	The dimension of this subspace is $\dim(W_i) = n_i - 1$.
		The dimension of this subspace is $\dim(W_i) = n_i - 1$. For any vector $x \in W_i$, we have $J_{n_i}x = 0$. Since $x$ is zero outside of $V_i$, all off-diagonal blocks of $DAD(G)$ annihilate $x$ as well. Multiplying $DAD(G)$ by $x$ yields
		\[
		DAD(G)x = \left((n - n_i - 2)I_{n_i} + 2J_{n_i}\right)x = (n - n_i - 2)x.
		\]
		Thus, $n - n_i - 2$ is an eigenvalue of $DAD(G)$ with multiplicity $n_i - 1$ for each part. This contributes the structural factor
		\[
		\prod_{i=1}^m (x - n + n_i + 2)^{n_i-1}
		\]
		to the characteristic polynomial.
		
	Step 3:
		The remaining eigenvalues correspond to eigenvectors that are constant on each partite set. Let $x = (c_1 \mathbf{1}_{n_1}^T, c_2 \mathbf{1}_{n_2}^T, \ldots, c_m \mathbf{1}_{m}^T)^T$, where $\mathbf{1}_{n_i}$ denotes the all-ones column vector of length $n_i$. For any vertex $v \in V_i$, the corresponding row execution yields
		\[
		(DAD(G)x)_v = (n - n_i - 2)c_i + 2n_i c_i + \sum_{j \ne i} n_j c_j = (n + n_i - 2)c_i + \sum_{j \ne i} n_j c_j.
		\]
		The induced quotient matrix $Q$ acting on this $m$-dimensional subspace is
		\[
		Q = \begin{pmatrix}
			n + n_1 - 2 & n_2 & \cdots & n_m \\
			n_1 & n + n_2 - 2 & \cdots & n_m \\
			\vdots & \vdots & \ddots & \vdots \\
			n_1 & n_2 & \cdots & n + n_m - 2
		\end{pmatrix}.
		\]
		The remaining $m$ eigenvalues of $DAD(G)$ are precisely the eigenvalues of $Q$.
		
Step 4: 
		We can express the characteristic matrix $xI_m - Q$ as a rank-one perturbation of a scalar matrix:
		\[
		xI_m - Q = (x - n + 2)I_m - \mathbf{1}_m r^T,
		\]
		where $\mathbf{1}_m = (1, 1, \ldots, 1)^T$ and $r = (n_1, n_2, \ldots, n_m)^T$. 
		Applying the matrix determinant lemma, $\det(D - uv^T) = \det(D)(1 - v^T D^{-1} u)$, with $D = (x - n + 2)I_m$, $u = \mathbf{1}_m$, and $v = r$, we obtain
		\[
		\det(xI_m - Q) = \det\left((x - n + 2)I_m\right) \left( 1 - r^T \frac{1}{x - n + 2}I_m \mathbf{1}_m \right).
		\]
		Evaluating the determinant and the inner product gives
		\[
		\det(xI_m - Q) = (x - n + 2)^m \left( 1 - \frac{1}{x - n + 2} \sum_{i=1}^m n_i \right).
		\]
		Since $\sum_{i=1}^m n_i = n$, this simplifies directly to
		\[
		\det(xI_m - Q) = (x - n + 2)^m \left( \frac{x - 2n + 2}{x - n + 2} \right) = (x - n + 2)^{m-1}(x - 2n + 2).
		\]
		
	Step 5:
		Combining the large multiplicity factors obtained in Step 2 with the determinant of the quotient matrix computed in Step 4, the complete characteristic polynomial is
		\[
		\Phi_{DAD(G)}(x) = (x - 2n + 2)(x - n + 2)^{m-1} \prod_{i=1}^m (x - n + n_i + 2)^{n_i-1}.
		\]
		This completes the proof.
	\end{proof}

\begin{cor}
	Let $K_{m,n}$, where $m,n\geq 2$, be a complete bipartite graph.
	Then the characteristic polynomial of its Degree--Adjacency--Diametrical
	matrix is
	
	\[
	\Phi_{DAD(K_{m,n})}(x)
	=
	(x-2m-2n+2)(x-m-n+2)
	(x-n+2)^{m-1}(x-m+2)^{n-1}.
	\]
	
	Consequently,
	
	\[
	\operatorname{Spec}(DAD(K_{m,n}))
	=
	\{(2m+2n-2)^1,(m+n-2)^1,
	(n-2)^{m-1},(m-2)^{n-1}\}.
	\]
	
\end{cor}

	\begin{cor}
	Let $S_n = K_{1,n-1}$ ($n \ge 2$) be the star graph on $n$ vertices, and let $DAD(S_n)$ denote its Degree–Adjacency–Diametrical matrix. Then
	\[
	\Phi_{DAD(S_n)}(x) = 
	\begin{cases}
		x(x-2), & n=2, \\[6pt]
		(x+1)^{n-2}(x-(2n-2))(x-(n-2)), & n \ge 3.
	\end{cases}
	\]
	Consequently,
	\[
	\operatorname{Spec}(DAD(S_n)) = 
	\begin{cases}
		\{\,2, 0\,\}, & n=2, \\[6pt]
		\{\,2n-2,\; n-2,\; (-1)^{(n-2)}\,\}, & n \ge 3,
	\end{cases}
	\]
	where $(-1)^{(n-2)}$ denotes the eigenvalue $-1$ with multiplicity $n-2$.
\end{cor}	

	\begin{thm}
	Let $G$ be a connected graph of order $n$ with diameter $2$. Let $\mu_1, \mu_2, \ldots, \mu_n$ be the eigenvalues of the Degree–Adjacency–Diametrical matrix $DAD(G)$. Then
	\[
	W(G) = \frac{1}{6}\left(2n(n-1) - \sum_{v \in V(G)} d(v)^2 + \sum_{i=1}^{n} \mu_i^2 \right).
	\]
\end{thm}

\begin{proof}
	Let $M = DAD(G) = (m_{ij})$. Since $M$ is a real symmetric matrix,
	\[
	\sum_{i=1}^{n} \mu_i^2 = \operatorname{tr}(M^2) = \sum_{i=1}^{n} \sum_{j=1}^{n} m_{ij}^2.
	\]
	Because $\operatorname{diam}(G) = 2$, the entries of $M$ are given by
	\[
	m_{ij} = 
	\begin{cases} 
		d(v_i), & i = j, \\ 
		1, & v_i \sim v_j, \\ 
		2, & v_i \nsim v_j, \ i \ne j. 
	\end{cases}
	\]
	Hence,
	\begin{itemize}
		\item the diagonal entries contribute $\sum_{v \in V(G)} d(v)^2$;
		\item there are $2m$ ordered adjacent pairs, each contributing $1^2 = 1$, so the total contribution is $2m$;
		\item there are $2\left(\binom{n}{2} - m\right)$ ordered nonadjacent pairs, each contributing $2^2 = 4$, so the total contribution is $8\left(\binom{n}{2} - m\right)$.
	\end{itemize}
	Therefore,
	\[
	\sum_{i=1}^{n} \mu_i^2 = \sum_{v \in V(G)} d(v)^2 + 2m + 8\left(\binom{n}{2} - m\right).
	\]
	Since $2\binom{n}{2} = n(n-1)$, we obtain
	\[
	\sum_{i=1}^{n} \mu_i^2 = \sum_{v \in V(G)} d(v)^2 + 4n(n-1) - 6m.
	\]
	Rearranging gives
	\[
	m = \frac{1}{6}\left(4n(n-1) + \sum_{v \in V(G)} d(v)^2 - \sum_{i=1}^{n} \mu_i^2 \right).
	\]
	Now, since $\operatorname{diam}(G) = 2$, every unordered pair of vertices is either adjacent or at distance $2$. Hence, the Wiener index $W(G)$ is
	\[
	W(G) = m + 2\left(\binom{n}{2} - m\right) = n(n-1) - m.
	\]
	Substituting the above expression for $m$, we obtain
	\[
	W(G) = n(n-1) - \frac{1}{6}\left(4n(n-1) + \sum_{v \in V(G)} d(v)^2 - \sum_{i=1}^{n} \mu_i^2 \right).
	\]
	Simplifying,
	\[
	W(G) = \frac{1}{6}\left(2n(n-1) - \sum_{v \in V(G)} d(v)^2 + \sum_{i=1}^{n} \mu_i^2 \right).
	\]
	This completes the proof.
\end{proof}
	
	\begin{thm}
		Let $G$ be a connected graph with $n$ vertices, $m$ edges, and diameter $D\geq 2$. Let $\mu_1, \mu_2, \ldots, \mu_n$ be the eigenvalues of the Degree–Adjacency–Diametrical matrix $DAD(G)$. Define
		\[
		\hat{d}(G) = \left|\{(u,v) \in V(G)^2 : u \neq v, \, d(u,v) = D\}\right|,
		\]
		that is, $\hat{d}(G)$ is the number of ordered pairs of vertices at distance $D$. Then
		\begin{align*}
			\sum_{i=1}^{n}\mu_i &= 2m, \\[1ex]
			\sum_{i=1}^{n}\mu_i^2 &= \sum_{j=1}^{n}d(v_j)^2 + 2m + D^2\hat{d}(G).
		\end{align*}
	\end{thm}
	
	\begin{proof}
		Since $DAD(G)$ is a real symmetric matrix, all its eigenvalues are real. Hence,
		\[
		\sum_{i=1}^{n}\mu_i = \operatorname{tr}(DAD(G)).
		\]
		By the definition of the $DAD$ matrix, the diagonal entry corresponding to the vertex $v_j$ is $d(v_j)$. Therefore,
		\[
		\operatorname{tr}(DAD(G)) = \sum_{j=1}^{n} d(v_j).
		\]
		Using the Handshaking Lemma,
		\[
		\sum_{j=1}^{n} d(v_j) = 2m.
		\]
		Consequently,
		\[
		\sum_{i=1}^{n}\mu_i = 2m.
		\]
		Next, because $DAD(G)$ is symmetric,
		\[
		\sum_{i=1}^{n}\mu_i^2 = \operatorname{tr}(DAD(G)^2).
		\]
		Let $A = (a_{ij}) = DAD(G)$. For any symmetric matrix, the trace of its square equals the sum of the squares of all individual entries:
		\[
		\operatorname{tr}(A^2) = \sum_{i=1}^{n}\sum_{j=1}^{n} a_{ij}^2.
		\]
		We now evaluate this total sum according to the definition of the $DAD$ matrix by partitioning the vertex pairs into four disjoint groups:
		\begin{itemize}
			\item \textbf{Diagonal entries ($i = j$):} Since $a_{ii} = d(v_i)$, their total contribution is:
			\[
			\sum_{i=1}^{n} a_{ii}^2 = \sum_{i=1}^{n} d(v_i)^2.
			\]
			\item \textbf{Adjacent vertices ($d(v_i, v_j) = 1$):} If $v_i$ and $v_j$ are adjacent, then $a_{ij} = 1$. There are exactly $2m$ ordered adjacent pairs, so their total contribution is $2m \cdot (1)^2 = 2m$.
			\item \textbf{Diametrical vertices ($d(v_i, v_j) = D$):} If the distance between $v_i$ and $v_j$ equals the diameter $D$, then $a_{ij} = D$. By definition, the total number of such ordered pairs is $\hat{d}(G)$, yielding a contribution of $D^2\hat{d}(G)$.
			\item \textbf{Remaining pairs ($1 < d(v_i, v_j) < D$):} For all remaining structural pairs, the matrix entries are defined as $a_{ij} = 0$, so they contribute nothing to the sum.
		\end{itemize}
		Summing all of these disjoint cases together, we obtain
		\[
		\operatorname{tr}(DAD(G)^2) = \sum_{j=1}^{n}d(v_j)^2 + 2m + D^2\hat{d}(G).
		\]
		Therefore,
		\[
		\sum_{i=1}^{n}\mu_i^2 = \sum_{j=1}^{n}d(v_j)^2 + 2m + D^2\hat{d}(G).
		\]
		This completes the proof.
	\end{proof}

\begin{cor}
	Let $G$ be a connected graph with $n$ vertices, $m$ edges, and
	diameter $D$. Let $\mu_1,\mu_2,\ldots,\mu_n$ be the eigenvalues of the
	Degree--Adjacency--Diametrical matrix $DAD(G)$. Then the DAD energy
	
	\[
	E_{DAD}(G)=\sum_{i=1}^{n}|\mu_i|
	\]
	
	satisfies
	
	\[
	E_{DAD}(G)\geq
	\sqrt{
		\sum_{v\in V(G)}d(v)^2+2m+D^2\widehat{d}(G)
	},
	\]
	
	where
	
	\[
	\widehat{d}(G)
	=
	|\{(u,v)\in V(G)\times V(G):u\neq v,\ d(u,v)=D\}|
	\]
	
	denotes the number of ordered pairs of vertices at distance $D$.
\end{cor}

\begin{proof}
	By the definition of the DAD energy,
	
	\[
	E_{DAD}(G)=\sum_{i=1}^{n}|\mu_i|.
	\]
	
	Since each $|\mu_i|$ is non-negative, we have
	
	\[
	\left(\sum_{i=1}^{n}|\mu_i|\right)^2
	\geq
	\sum_{i=1}^{n}|\mu_i|^2
	=
	\sum_{i=1}^{n}\mu_i^2.
	\]
	
	By Theorem~2.8,
	
	\[
	\sum_{i=1}^{n}\mu_i^2
	=
	\sum_{v\in V(G)}d(v)^2
	+
	2m+
	D^2\widehat{d}(G).
	\]
	
	Therefore,
	
	\[
	E_{DAD}(G)^2
	\geq
	\sum_{v\in V(G)}d(v)^2
	+
	2m+
	D^2\widehat{d}(G).
	\]
	
	Taking the positive square root on both sides, we obtain
	
	\[
	E_{DAD}(G)\geq
	\sqrt{
		\sum_{v\in V(G)}d(v)^2+2m+D^2\widehat{d}(G)
	}.
	\]
	
	Hence the result follows.
\end{proof}

	\begin{thm}
		Let $G$ be a connected graph with $n$ vertices, $m$ edges, and diameter $D$. Let $\mu_1 \ge \mu_2 \ge \cdots \ge \mu_n$ be the eigenvalues of the Degree–Adjacency–Diametrical matrix $DAD(G)$. Then
		\[
		\mu_1 \le \frac{2m + \sqrt{(n-1)\left[n\left(\sum_{v\in V(G)}d(v)^2 + 2m + D^2\hat{d}(G)\right) - 4m^2\right]}}{n},
		\]
		where $\hat{d}(G)$ denotes the number of ordered pairs of vertices at distance $D$.
	\end{thm}
	
	\begin{proof}
		Since $DAD(G)$ is a real symmetric matrix, all its eigenvalues are real. Moreover, the sum of the eigenvalues equals the trace of $DAD(G)$. Since the diagonal entries of $DAD(G)$ are the vertex degrees, we have
		\[
		\sum_{i=1}^{n}\mu_i = \operatorname{tr}(DAD(G)) = \sum_{v\in V(G)}d(v) = 2m,
		\]
		by the Handshaking Lemma.
		
		Again, since $DAD(G)$ is symmetric,
		\[
		\sum_{i=1}^{n}\mu_i^2 = \operatorname{tr}\left(DAD(G)^2\right).
		\]
		The diagonal entries of $DAD(G)^2$ are obtained by squaring the entries of each row of $DAD(G)$. Hence,
		\[
		\operatorname{tr}\left(DAD(G)^2\right) = \sum_{v\in V(G)}d(v)^2 + 2m + D^2\hat{d}(G),
		\]
		where $2m$ is the contribution from the adjacent vertex pairs and $D^2\hat{d}(G)$ is the contribution from the ordered pairs of vertices at distance $D$. Therefore,
		\[
		\sum_{i=1}^{n}\mu_i^2 = \sum_{v\in V(G)}d(v)^2 + 2m + D^2\hat{d}(G).
		\]
		Now, we can isolate the first eigenvalue:
		\[
		\sum_{i=2}^{n}\mu_i = 2m - \mu_1.
		\]
		Applying the Cauchy–Schwarz inequality to the $n-1$ numbers $\mu_2, \mu_3, \dots, \mu_n$, we obtain
		\[
		\left(\sum_{i=2}^{n}\mu_i\right)^2 \le (n-1)\sum_{i=2}^{n}\mu_i^2.
		\]
		Hence,
		\[
		(2m - \mu_1)^2 \le (n-1)\sum_{i=2}^{n}\mu_i^2.
		\]
		Since $\sum_{i=2}^{n}\mu_i^2 = \sum_{i=1}^{n}\mu_i^2 - \mu_1^2$, it follows that
		\[
		(2m - \mu_1)^2 \le (n-1)\left(\sum_{v\in V(G)}d(v)^2 + 2m + D^2\hat{d}(G) - \mu_1^2\right).
		\]
		Expanding the left-hand side yields
		\[
		4m^2 - 4m\mu_1 + \mu_1^2 \le (n-1)\left(\sum_{v\in V(G)}d(v)^2 + 2m + D^2\hat{d}(G)\right) - (n-1)\mu_1^2.
		\]
		Collecting all terms involving $\mu_1$ on the left side of the inequality, we get
		\[
		n\mu_1^2 - 4m\mu_1 + 4m^2 - (n-1)\left(\sum_{v\in V(G)}d(v)^2 + 2m + D^2\hat{d}(G)\right) \le 0.
		\]
		This is a quadratic inequality in $\mu_1$. Since the coefficient of $\mu_1^2$ is positive ($n > 0$), $\mu_1$ must be bounded above by the larger root of the corresponding quadratic equation. Consequently,
		\[
		\mu_1 \le \frac{2m + \sqrt{4m^2 - n\left[4m^2 - (n-1)\left(\sum_{v\in V(G)}d(v)^2 + 2m + D^2\hat{d}(G)\right)\right]}}{n}.
		\]
		Using the algebraic identity
		\[
		4m^2 - n\left[4m^2 - (n-1)S\right] = (n-1)(nS - 4m^2),
		\]
		where $S = \sum_{v\in V(G)}d(v)^2 + 2m + D^2\hat{d}(G)$, we can simplify the expression under the square root to obtain:
		\[
		\mu_1 \le \frac{2m + \sqrt{(n-1)\left[n\left(\sum_{v\in V(G)}d(v)^2 + 2m + D^2\hat{d}(G)\right) - 4m^2\right]}}{n}.
		\]
		This completes the proof.
	\end{proof}

\section{Conclusion}

In this paper, we introduced the Degree--Adjacency--Diametrical (DAD)
matrix, a new graph matrix that combines local connectivity information
through adjacency and vertex degrees with global extremal distance
information through graph diameter. Inspired by recent structural
variations such as the adjacency--diametrical matrix, the proposed matrix
provides a unified framework for investigating hybrid spectral properties
of graphs. We derived explicit characteristic polynomials and complete
spectra for several important graph classes, including paths, cycles,
complete graphs, and complete multipartite graphs. Furthermore, we
established trace identities, obtained lower bounds for the DAD energy,
and explored the relationship between the DAD spectrum and the Wiener index
for graphs with diameter two.

\end{document}